\documentclass[11pt]{article}
\usepackage{amssymb,amsmath,amsfonts,amsthm,mathtools,color}

\usepackage{mathrsfs}
\usepackage{dsfont}
 \usepackage{bm} 

\usepackage{comment} 
\usepackage[title,titletoc,header]{appendix}
\usepackage{graphicx,subfig}
\usepackage{float} 
\usepackage{algorithm,algorithmic}

\usepackage{empheq} 

\usepackage{paralist}

\usepackage{tikz}
\usetikzlibrary{arrows,positioning,shapes.geometric}

\graphicspath{ {./figures/} }

\usepackage{indentfirst}
\usepackage{multicol}
\usepackage{booktabs}
\usepackage{url}
\usepackage[outdir=./]{epstopdf} 

\usepackage[shortlabels]{enumitem}

\usepackage{hyperref}
\hypersetup{
    colorlinks=true, 
    linktoc=all,     
    linkcolor=blue,  
}
\numberwithin{equation}{section}

\newtheorem{Theorem}{Theorem}[section]

\newtheorem{Proposition}[Theorem]{Proposition}
\newtheorem{Corollary}[Theorem]{Corollary}

\theoremstyle{definition}
\newtheorem{Definition}{Definition}[section]

\theoremstyle{remark}
\newtheorem{Remark}{Remark}[section]

\def\to{\rightarrow}

\def\ms{\medskip}

\def\cF{\mathcal{F}}

\def\d{{\mathrm{d}}}

\def\sE{{\mathbb{E}}}

\def\sN{{\mathbb{N}}}
\def\sP{\mathbb{P}}

\def\sR{{\mathbb R}}

\newcommand{\lc}
{\mathrel{\raise2pt\hbox{${\mathop<\limits_{\raise1pt\hbox
{\mbox{$\sim$}}}}$}}}

\newcommand{\gc}
{\mathrel{\raise2pt\hbox{${\mathop>\limits_{\raise1pt\hbox{\mbox{$\sim$}}}}$}}}

\newcommand{\ec}
{\mathrel{\raise2pt\hbox{${\mathop=\limits_{\raise1pt\hbox{\mbox{$\sim$}}}}$}}}

\def\bb{\begin{equation}} \def\ee{\end{equation}}
\def\bbn{\begin{equation*}} \def\een{\end{equation*}}

\def\beqn{\begin{eqnarray}}  \def\eqn{\end{eqnarray}}

\def\beqnx{\begin{eqnarray*}} \def\eqnx{\end{eqnarray*}}

\def\bn{\begin{enumerate}} \def\en{\end{enumerate}}

\def\bd{\begin{description}} \def\ed{\end{description}}



\definecolor{DarkGreen}{rgb}{0.2,0.6,0.2}

\def\pink#1{\textcolor{\pink}{#1}}

\definecolor{brilliantrose}{rgb}{1.0, 0.33, 0.64}

\begin{document}

\title{
$\alpha$-Potential Games
for Decentralized Control of Connected and Automated Vehicles
  }

\author{
Xuan Di\thanks{Department of Civil Engineering and Engineering Mechanics, Data Science Institute, Columbia University}
\and
Anran Hu\thanks{Department of Industrial Engineering and Operations Research, Columbia University}
\and
Zhexin Wang\footnotemark[2]
\and 
Yufei Zhang\thanks{Department of Mathematics, Imperial College London,  London,  UK} 
}




\date{}
\maketitle

\begin{abstract}
   Designing scalable and safe control strategies for large populations of connected and automated vehicles (CAVs) requires accounting for strategic interactions among heterogeneous agents under decentralized information. While dynamic games provide a natural modeling framework, computing Nash equilibria (NEs) in large-scale settings remains challenging, and existing mean-field game approximations rely on restrictive assumptions that fail to capture collision avoidance and heterogeneous behaviors.
This paper proposes an $\alpha$-potential game framework for decentralized CAV control. We show that computing $\alpha$-NE reduces to solving a decentralized control problem, and derive tight bounds of the parameter $\alpha$ based on interaction intensity and asymmetry. We further develop scalable policy gradient algorithms for computing $\alpha$-NEs using decentralized neural-network policies. Numerical experiments demonstrate that the proposed framework accommodates diverse traffic flow models and effectively captures collision avoidance,
obstacle avoidance, 
and agent heterogeneity.
 
\end{abstract}


\ms



\medskip


\section{Introduction}
\label{sec:intro}

How can we design scalable control strategies for large populations of connected and automated vehicles (CAVs) that account for strategic interactions among agents with heterogeneous preferences while ensuring safety?
Addressing this question is essential to realizing the full potential of CAV systems.
Scalability demands decentralized policies that operate using local information, thereby avoiding the communication and computational burdens of centralized coordination. Meanwhile, these policies must also guarantee safety, with collision avoidance remaining a primary constraint in dynamic traffic environments \cite{ammour2021collision}. 

CAV control problems, such as platooning  \cite{di2021survey,braiteh2024platooning}
and lane change 
\cite{kamal2024cooperative}, 
have been studied extensively  from an optimal control perspective. 
Most of them rely on centralized control, where vehicles coordinate toward a shared objective using consensus-based methods \cite{bailo2018optimal} 
or model predictive control \cite{daini2024traffic}. 
However, centralized coordination typically requires substantial communication infrastructure, 
which may be impractical and costly  in real-world deployments.
While decentralized or distributed optimal control offers a more scalable and implementable alternative \cite{dai2017distributed}, it still assumes full cooperation among vehicles and cannot fully capture the strategic interactions or conflicts that naturally arise in heterogeneous traffic environments.


A dynamic game framework offers a natural approach for CAVs with heterogeneous and competing objectives (see e.g., \cite{jing2024decentralized,di2025mfgreview, yan2025markov}). 
Unlike traditional optimal control approaches that assume full cooperation toward a centralized goal, 
the game-theoretic approach treats each CAV as a rational agent optimizing its own self-interested payoff while strategically interacting with others. 
Agents’ behaviors are characterized by the Nash equilibrium (NE), a set of policies in which no agent can improve their   payoff by unilaterally deviating. 
A key step in this framework is to exploit structural properties of   multi-agent interactions to make NE computation tractable in large-scale dynamic games.

An established approach to obtaining equilibrium policies in large CAV games \cite{festa2018mean,huang2020game,di2025mfgreview} is to model vehicles as homogeneous agents that interact weakly through their empirical distribution and to solve the resulting mean-field game (MFG) formulation \cite{huang2006Largepopulation,lasry2007mean}. 
While intended to simplify the analysis for large population games, this approach relies on restrictive assumptions that limit its practical applicability in traffic control.
First, the MFG regime  implies that each vehicle is influenced only by the   population distribution rather than by specific nearby vehicles, making it unsuitable for modeling collision avoidance where close-proximity interactions play a dominant role. Second, the homogeneity assumption neglects heterogeneous objectives and asymmetric interactions that are inherent to mixed traffic environments. Third, despite the approximation, solving   MFGs can be computationally intractable  \cite{yardim2024mean}. These limitations indicate that   MFGs are inadequate   for capturing strategic interactions among heterogeneous CAVs.

Meanwhile, the recently proposed $\alpha$-potential game framework offers a powerful approach for   solving general dynamic games \cite{guo2025markov, guo2025alpha, guo2025distributed}. It reduces the computation of NEs of a dynamic game to optimizing a single $\alpha$-potential function and analyzing the parameter $\alpha$.  Unlike MFGs which rely on weak interactions among infinitely many homogeneous players, the $\alpha$-potential framework applies directly to finite-player settings with heterogeneous agents and potentially strong, local interactions.

\paragraph{Contributions.}
This paper adapts the $\alpha$-potential game framework to CAV control. In contrast to \cite{guo2025alpha, guo2025markov, guo2025distributed}, which focus on  games with open-loop controls depending on  system noises, we consider dynamic games with decentralized closed-loop policies, enabling scalable control of large-scale CAV systems. The main contributions are:
(i) We show that finding an $\alpha$-NE reduces to solving a   decentralized control problem. Due to the use of closed-loop policies, constructing this control problem requires techniques distinct from those in \cite{guo2025alpha, guo2025markov, guo2025distributed} for open-loop games (see Remark \ref{eq:potential}).
(ii)  We characterize the parameter $\alpha$ in terms of the intensity and asymmetry of vehicles’ interactions. The  bound is tighter than those in existing works, particularly for irregular interaction kernels.
(iii) We develop scalable policy gradient algorithms for decentralized neural-network–based policies. Extensive numerical experiments demonstrate that our framework captures collision avoidance,
obstacle avoidance,
and player heterogeneity more effectively than   MFG  approaches.

\section{Markov   games for CAV control}
\label{sec:problem}
This section formulates a  dynamic Markov game with decentralized policies  for CAV control.

\paragraph{Problem Setup.}
Consider a finite-player stochastic differential game with decentralized Markov policies, defined as follows: 
$T>0$ is a fixed  time horizon, 
   $[N]=\{1,\ldots, N\}$,  $N\in \mathbb N$,
    is a finite set of
players, and for all $i\in [N]$, 
$ A_i\subset \sR^{k}$ is  player $i$'s action set,
and 
     $\pi_i$ is the set  of player $i$'s admissible policies, containing   measurable functions $\phi_i:[0,T]\times \sR^{d}\to A_i$ which satisfy   
 $\sup_{t\in [0,T]}|\phi_i(t,0)|<\infty$ and  are Lipschitz continuous in the second variable.
 We denote by 
$\pi=\prod_{i\in [N]}\pi_i$   the set of all admissible policy profiles of all players, and by   
$\pi_{-i}=\prod_{j\not = i}\pi_j$   the set of admissible policy profiles of all players except player $i$. We denote by 
 $\phi=(\phi_i)_{i\in [N]}$
 and 
 $\phi_{-i}=(\phi_j)_{j\in [N]\setminus\{i\}}$
 a generic element of $\pi$ and $\pi_{-i}$, respectively.  

Each player’s state evolves according to a controlled diffusion, determined by their chosen policy. The player then minimizes a cost functional over the set of admissible policies. 
More precisely, 
for any given $\phi=(\phi_i)_{i\in [N]}\in \pi$,
player $i$'s state dynamics $X^{\phi_i}_i$ 
is governed by  the following  dynamics:
for all $t\in [0,T]$,
\begin{equation}
\label{eq:state_i}
    \d  X_{i,t} = b_i(X_{i,t},\phi_i(t,X_{i,t}))
\d t +\sigma_i  (X_{i,t},\phi_i(t,X_{i,t}))\d W_t, \quad X_{i,0}=\xi_i,
\end{equation}
where $\xi_i\in \sR^d$ is a given initial state, 
 $b_i: \sR^{d}\times A_i\to \sR^{d}$
 and 
 $\sigma_i: \sR^{d}\times A_i\to \sR^{d\times m}$
 are   Lipschitz continuous functions,   and  
 $W:\Omega\times [0,T]\to \sR^m$ is  an  $m$-dimensional Brownian motion  on 
a complete  probability space 
$(\Omega,\cF,\sP)$. Given $\phi_{-i}\in \pi_{-i}$,
 player $i$ determines their optimal policy  by minimizing the following objective function   over  $\pi_i$:
\begin{equation}
\label{eq:cost_i}
      J_i(\phi) = \sE \Biggl[\int_0^T 
  \Biggl(
  f_i\big(
  X^{\phi_i}_{i,t},
  \phi_{i}(t,X^{\phi_i}_{i,t})\big) +\sum_{j\not =i}\lambda_{ij}K(X^{\phi_i}_{i,t}-X^{\phi_j}_{j,t})
  \Biggr)\d t  + g_i(X^{\phi_i}_{i,T})\Biggl],
\end{equation}
where 
$\lambda_{ij}\ge 0$ is a given constant, 
$f_i:\sR^d\times   A_i\to \sR$ and 
$K:\sR^d\to \sR$ 
are  given    running costs,
and 
 $g_i:\sR^d\to\sR$ is  a given  terminal cost.
 We assume that
 $f_i$ and $g_i$ are at most of quadratic growth, 
 and 
 $K$ is a bounded function satisfying  
   $K(x)=K(-x)$ for all $x\in \sR^d$.

 We  characterize the rational behaviour of players in the Markov game \eqref{eq:state_i}–\eqref{eq:cost_i} using the notion of an $\epsilon$-Nash equilibrium given below.

      \begin{Definition}
     \label{def:NE}
          For any $\epsilon\ge 0$, a policy profile 
          $\bar{\phi}=(\bar{\phi}_i)_{i\in [N]}\in \pi$  an $\epsilon$-Nash equilibrium (NE) of the
game \eqref{eq:state_i}–\eqref{eq:cost_i} if
$J_i(\bar{\phi})\le 
J_i(\phi_i,\bar{\phi}_{-i})+\epsilon$, for all $i\in [N]$ and $\phi_i\in \pi_i$. 

     \end{Definition}

According to Definition \ref{def:NE}, a joint  policy profile $\bar \phi$ is an $\epsilon$-NE if no player can improve her objective by more than $\epsilon$ through any unilateral deviation. Note that although player $i$'s objective in \eqref{eq:cost_i} depends on the joint states of all players, player $i$ optimizes \eqref{eq:cost_i} over decentralized policies that depend only on her own private state $X_i$.

\paragraph{Application to CAV Control.} The game \eqref{eq:state_i}-\eqref{eq:cost_i} encompasses various dynamic games    in CAV control. In this setting, each player represents a vehicle. The dynamics   \eqref{eq:state_i} can be taken as controlling velocity: 
\begin{equation}
\label{eq:CAV_1}
\d  x_{i,t} =
 \phi_i(t,x_{i,t} )
\d t, \quad t\in [0,T],
\end{equation}
where $x_{i,t}$  represents 
  the position   of the $i$-th vehicle at time $t$, 
  and the control 
  $v_{i,t}\coloneqq \phi_i(t,x_{i,t})$ represents 
  the   velocity    of the $i$-th vehicle at time $t$. 

Alternatively, 
one can take the dynamics   \eqref{eq:state_i} as a double integrator with controlled acceleration:
for all $t\in [0,T]$,
\begin{align}
\label{eq:CAV_2}
    \begin{split}
\d  x_{i,t} &=
 v_{i,t} \d t,
\quad 
\d  v_{i,t}  =
 \phi_i(t,x_{i,t}, v_{i,t})
\d t +\sigma_i \d W_t, 
    \end{split}
\end{align}

where $x_{i,t}$ and $v_{i,t}$  represent
  the position  and velocity   of the $i$-th vehicle at time $t$, respectively,
  and  the control 
  $u_{i,t}\coloneqq \phi_i(t,x_{i,t},v_{i,t})$ represents 
  the   acceleration    of the $i$-th vehicle at time $t$. 
  
  Vehicle $i$ determines its optimal route by minimizing the objective  \eqref{eq:cost_i},
where 
the terminal cost $  g_i$  specifies vehicle $ i$'s preferred target,    the running  cost $ f_i$ penalizes control effort  and
 enforces  the obstacle-avoidance behaviour for vehicle $ i $, and  the   kernel $ K $ together with the weights $ {(\lambda_{ij})}_{j\not = i} $  characterizes how the spatial distribution of other vehicles  influences vehicle $i$'s preferred route.
Typically, the kernel 
$K$
 decreases as the distance between two vehicles increases, modeling congestion-averse behavior.

\paragraph{Collision Avoidance Beyond Mean-Field.}
Our model captures local interactions and collision-avoidance behavior, which are crucial for CAV control and cannot be adequately represented within the traditional MFG framework \cite{huang2020game}; See Section \ref{sec:numerical_interaction} for details.  
This can be achieved by allowing the  weights $\lambda_{ij}$ to remain non-vanishing as the number of players $N \to \infty$, and  the  kernel $K$ to depend explicitly on $N$. 

For instance,   consider  the following $\lambda_{ij}$ and $K$ as in \cite{oelschlager1985law}: 
\begin{equation}
\label{eq:interaction_kernel}
\lambda_{ij}=N^{\beta-1}, \quad K(z)= \rho(N^{\beta/d}|z|),
\quad z\in \sR^d,
\end{equation}
where  $\rho:[0,\infty)\to [0,\infty)$ satisfies   $\lim_{|z|\to \infty}\rho(z)= 0$,
and $\beta\in [0,1]$ determines how the interaction strength among players scales with $N$.
When $\beta =0$, 
each player interacts with all others with strength $\mathcal O(1/N)$. This is  
the classical weak-interaction regime   assumed to facilitate mean-field approximations \cite{huang2006Largepopulation,lasry2007mean}.
 As $N\to \infty$, each player is influenced only by the aggregate population distribution rather than by individual players, as in the  MFG  formulation for traffic flow in \cite{huang2020game}.
 In contrast, 
when $\beta = 1$, the decay of $\rho$ implies that players interact only with neighbors at distance $\mathcal O(N^{-1/d})$, but with interaction strength of order one. This strong-interaction  regime models CAV systems more realistically, since   nearby vehicles exert non-negligible influence on individual decisions.
In this case,
 the mean-field approximation fails, but our proposed framework remains applicable.

\section{Main Theoretical Results}
\label{sec:theory}

This section analyzes the approximate NEs of the game \eqref{eq:state_i}–\eqref{eq:cost_i} using the $\alpha$-potential game framework developed in \cite{guo2025alpha}. This framework reduces  the challenging problem of identifying (approximate) NEs into a single optimization problem of the associated $\alpha$-potential function.  
This connection is made precise in the following proposition, which follows directly  from  \cite[Proposition 2.1]{guo2025alpha}.

 \begin{Proposition}

  Suppose that 
  the game \eqref{eq:state_i}-\eqref{eq:cost_i} 
  is an $\alpha$-potential game for some $\alpha\ge 0$, in the sense that there exists 
  $\Phi:\pi\to \sR$,
  called an $\alpha$-potential function, 
  such that 
for all $i\in [N]$, $\phi_{-i}\in \pi_{-i}$, and
$\phi_i,\phi'_i\in \pi_i$, 
$$
|(J_i(\phi'_i,\phi_{-i})
-J_i(\phi_i,\phi_{-i}))
-(\Phi(\phi'_i,\phi_{-i})
-\Phi(\phi_i,\phi_{-i}))|\le \alpha.
$$
Then 
for all $\epsilon\ge 0$,
if $\bar \phi\in \pi$
satisfies 
$\Phi(\bar \phi)\le \inf_{\phi\in \pi}\Phi(\phi)+ \epsilon$,
then $\bar \phi$ is an $(\alpha+\epsilon)$-NE of the game \eqref{eq:state_i}-\eqref{eq:cost_i}.
 
\end{Proposition}

   The  following theorem constructs analytically an $\alpha$-potential function
for the game \eqref{eq:state_i}-\eqref{eq:cost_i} and quantifies the associated $\alpha$. 

\begin{Theorem}
\label{thm:alpha_PG}
 
Define $\Phi:\pi\to \sR$ such that for all $\phi\in \pi$,
\begin{equation}\label{eq:potential_fun_symmetric}
     \Phi(\phi) \coloneqq \sE \left[\int_0^T F(X^{\phi}_t,\phi(t, X^{\phi}_t))\d t + G(X^{\phi}_T)\right],
\end{equation}
where
$X^\phi=(X^{\phi_i}_i)_{i\in [N]}$
is the joint state process,
$\phi_t(t, X^{\phi}_t)
=
(\phi_i(t, X^{\phi_i}_{i,t}))_{i\in [N]}$ 
is the joint control process, and 
$F:\sR^{dN}\times A\to \sR$
and $G:\sR^{dN} \to \sR$
are given by
\begin{align*}
\begin{split}
F(x,a) \coloneqq \sum_{i=1}^N f_i(x_i,a_i) +
\sum_{1\le i<j\le N} \frac{\lambda_{ij}+\lambda_{ji}}{2} K(x_i-x_j), \quad
G(x)   \coloneqq \sum_{i=1}^N g_i(x_i). 
\end{split}
\end{align*}
Then $\Phi$ is an $\alpha$-potential function of the game  \eqref{eq:state_i}-\eqref{eq:cost_i}  with 
\begin{equation}
\label{eq:alpha_bound}
    \alpha \leq T\|K\|_{L^\infty}\max_{i\in [N]}\sum_{i\neq j}|\lambda_{ij}-\lambda_{ji}|.
\end{equation}
Hence any minimizer of $\Phi$ is an $\alpha$-NE of   the game  \eqref{eq:state_i}-\eqref{eq:cost_i}. 
\end{Theorem}

 Theorem 
    \ref{thm:alpha_PG} 
    characterizes  the magnitude of $\alpha$  in terms of the   time horizon, the magnitude of the interaction kernel $K$  and the   weights $(\lambda_{ij})_{i,j
    \in [N]} $. 
     Clearly, $\alpha=0$ if the interaction is symmetric (i.e., $
    \lambda_{ij}=\lambda_{ji})$. We   refer to \cite{guo2025alpha,guo2025distributed}
for sufficient conditions under which $
\alpha$ vanishes as the number of players tends to 
infinity. 
Note that 
the upper bound  \eqref{eq:alpha_bound} depends on $K$ itself rather than its second-order derivative as in \cite{guo2025alpha,guo2025distributed}, and is therefore tighter for irregular kernels.
\begin{Remark}
   \label{eq:potential}
   
 The construction of the $\alpha$-potential function in Theorem \ref{thm:alpha_PG} differs from that in \cite{guo2025alpha, guo2025distributed} for stochastic games with \emph{open-loop} controls in the following aspects: 
(i) Theorem \ref{thm:alpha_PG} considers the Markov game \eqref{eq:state_i}–\eqref{eq:cost_i} with Lipschitz policies,
for which  the sensitivity-process techniques   in \cite{guo2025alpha,guo2025distributed} are inapplicable 
due to the non-differentiability of the   policies. 
(ii) 
Theorem \ref{thm:alpha_PG}  constructs an 
  $ \alpha$-potential function   \eqref{eq:potential_fun_symmetric}     using only   the state process    $X^\phi$, by exploiting the independence of 
$(X^{\phi_i}_i)_{i\in [N]}$.  This   reduces the computational cost of finding  an $ \alpha$-NE, 
  since the $\alpha$-potential function \eqref{eq:potential_fun_symmetric}  involves a $dN$-dimensional state process compared with the $2dN$-dimensional system required in \cite{guo2025distributed}.

\end{Remark}

\begin{proof}
    For all  $i\in[N]$, define a new objective  of player $i$:
\begin{equation}
\label{eq:cost_i_symmetric}
J^s_i(\phi) = \sE \Biggl[\int_0^T 
  \Biggl(
  f_i\big(
  X^{\phi_i}_{i,t},
  \phi_{i}(t,X^{\phi_i}_{i,t})\big) 
   +\sum_{j\not =i}\frac{\lambda_{ij}+\lambda_{ji}}{2}K(X^{\phi_i}_{i,t}-X^{\phi_j}_{j,t})
  \Biggr)\d t  + g_i(X^{\phi_i}_{i,T})\Biggl].
\end{equation}
By \cite[Theorem 3.1]{guo2025towards}, the game with the objectives $(J^s_i)_{i\in [N]}$ is a potential game with the potential function $\Phi$ defined in \eqref{eq:potential_fun_symmetric}, which implies
\[
\Phi(\phi_i,\phi_{-i})-\Phi(\phi_i',\phi_{-i}) = J_i^s(\phi_i,\phi_{-i})-J^s_i(\phi_i',\phi_{-i}).
\]
Note that
\begin{align*}
    |J_i^s(\phi)-J_i(\phi)|
    \leq \sE\left|\int_0^T\sum_{j\neq i}\left(\frac{1}{2}(\lambda_{ij}+\lambda_{ji})-\lambda_{ij}\right)K(X^{\phi_i}_{i,t}-X^{\phi_j}_{j,t})\d t \right|
    \leq \frac{T\|K\|_{L^\infty}}{2}\sum_{i\neq j}|\lambda_{ij}-\lambda_{ji}|.
\end{align*}
The desired conclusion then follows from
\begin{align*}
    &\left|\left(J_i(\phi_i,\phi_{-i})-J_i(\phi_i',\phi_{-i})\right)-\left(\Phi(\phi_i,\phi_{-i})-\Phi(\phi_i',\phi_{-i})\right)\right|\\
    =& \left|\left(J_i(\phi_i,\phi_{-i})-J_i(\phi_i',\phi_{-i})\right)-\left(J_i^s(\phi_i,\phi_{-i})-J^s_i(\phi_i',\phi_{-i})\right)\right|\\
    \leq& \left|J_i(\phi_i,\phi_{-i})-J_i^s(\phi_i,\phi_{-i})\right|+\left|J_i(\phi_i',\phi_{-i})-J^s_i(\phi_i',\phi_{-i})\right| \\
    \leq& T\|K\|_{L^\infty}\sum_{i\neq j}|\lambda_{ij}-\lambda_{ji}|.
\end{align*}
\end{proof}

One can further exploit specific structures of the interaction weights $\lambda_{ij}$  to construct a refined potential function with a reduced $\alpha$. The following corollary illustrates this for the separable weights $\lambda_{ij}=\gamma_i\tau_j$, where $\gamma_i$ represents player $i$’s sensitivity to others, and $\tau_j$ captures the influence of player $j$ on others.
In this case, the game   \eqref{eq:state_i}-\eqref{eq:cost_i} is equivalent to a potential game. 

\begin{Corollary}
   Suppose that 
   for all $i,j\in [N]$,  $\lambda_{ij}=\gamma_i\tau_j$
   for some $\gamma_i,\tau_j>0$.
   Define the function 
   $\tilde\Phi:\pi\to \sR$ such that for all $\phi\in \pi$,
\begin{equation}
\label{eq:potential_fun_rescale}
     \tilde\Phi(\phi) \coloneqq \sE \left[\int_0^T \tilde F(X^{\phi}_t,\phi(t, X^{\phi}_t))\d t + \tilde G(X^{\phi}_T)\right],
\end{equation}
 where 
$X^\phi_t$ and $\phi(t,X^\phi_t)$
are defined as in \eqref{eq:potential_fun_symmetric},
and 
\begin{align}
\begin{split}
\tilde F(x,a)& \coloneqq \sum_{i=1}^N \frac{\tau_i}{\gamma_i}f_i(x_i,a_i)+\frac{1}{2}\sum_{i\neq j}d_id_jK(x_i-x_j),
\\
\tilde G(x) & \coloneqq \sum_{i=1}^N \frac{\tau_i}{\gamma_i}g_i(x_i).
\end{split}
\end{align}
Then any minimizer of 
\eqref{eq:potential_fun_rescale}
is an NE of the   game 
\eqref{eq:state_i}-\eqref{eq:cost_i}. 
\end{Corollary}

\begin{proof}
   Observe that the set of NEs of the   game 
\eqref{eq:state_i}-\eqref{eq:cost_i} coincdes with the set of NEs of the game with the  scaled objectives $(\tilde J_i)_{i\in [N]}$ defined by
     \begin{equation}\label{eq:cost_i_rescale}
    \tilde J_i(\phi) = \frac{\tau_i}{\gamma_i} J_i(\phi)= \sE \Biggl[\int_0^T 
  \Biggl(\frac{\tau_i}{\gamma_i} f_i\big(
  X^{\phi_i}_{i,t},
  \phi_{i}(t,X^{\phi_i}_{i,t})\big)+\sum_{j\not =i}\tau_i\tau_jK(X^{\phi_i}_{i,t}-X^{\phi_j}_{j,t})
  \Biggr)\d t  + \frac{\tau_i}{\gamma_i}g_i(X^{\phi_i}_{i,T})\Biggr].
     \end{equation}
By Theorem \ref{thm:alpha_PG},
$\tilde \Phi$ is a  potential function 
of the game with 
 the  scaled objectives $(\tilde J_i)_{i\in [N]}$,
 which completes the proof.  
\end{proof}

\section{Policy gradient algorithm for $\alpha$-NE}
\label{sec:algo}

This section proposes a policy gradient (PG) algorithm to  minimize the $\alpha$-potential function  \eqref{eq:potential_fun_symmetric}, which subsequently yields an $\alpha$-NE of the   Markov   game \eqref{eq:state_i}–\eqref{eq:cost_i}.

The algorithm parametrizes player $i$'s  policy  using  a sufficiently expressive neural network  $\phi_i^{\theta_i}:[0,T]\times\sR^{d}\to A_i$ with parameters $\theta_i\in\sR^{L_i}$,
and seeks the optimal parameters $\theta=(\theta_i)_{i\in [N]}$ by minimizing  the following approximate
$\alpha$-potential function over $\theta$: 
\begin{equation}\label{eq:param-potential}
  \Phi(\theta):=\sE\left[\int_0^T F\!\big(X_t^\theta,\phi^\theta(t,X_t^\theta)\big)\d t+G\big(X_T^\theta\big)\right],
\end{equation}
where    $X_t^\theta=(X_{i,t}^{{\theta_i}})_{i\in [N]}$ with $X^{{\theta_i}}$ satisfying
\begin{equation}\label{eq:param-dynamics}
   \d X_{i,t}  = b_i(X_{i,t},\phi_i^{\theta_i}(t,X_{i,t}))\d t+\sigma_i(X_{i,t},\phi_i^{\theta_i}(t,X_{i,t}))\d W_t, 
\end{equation}
and  $\phi^\theta(t,X_t^\theta)=(\phi_i^{\theta_i}(t,X_{i,t}^{\theta_i}))_{i\in[N]}$. 
We find the minimizer of   the   function \eqref{eq:param-potential}   using  gradient-based   algorithms (e.g., the  Adam method). 

To evaluate the objective \eqref{eq:param-potential} and its derivative, we   discretize
\eqref{eq:param-dynamics} on time grid $\Pi_P=\{0=t_0<\dots<t_P=T\}$ for some $P\in \sN$: for any $i\in[N]$, $X^{\theta_i}_{i,0}=\xi_i$, and for any $\ell=0,\dots,P-1$, 
\begin{equation}\label{eq:param-dynamics-disc}
        X_{i,t_{\ell+1}}^{\theta_i} = X_{i,t_{\ell}}^{\theta_i} + b_i( X_{i,t_{\ell}}^{\theta_i}, \phi_i^{\theta_i}(t_\ell, X_{i,t_{\ell}}^{\theta_i}))\Delta_\ell      
        + \sigma_i( X_{i,t_{\ell}}^{\theta_i}, \phi_i^{\theta_i}(t_\ell, X_{i,t_{\ell}}^{\theta_i}))\Delta W_\ell,
\end{equation}
where 
$\Delta_\ell =t_{\ell+1}-t_\ell$, and $ 
\Delta W_\ell = W_{t_{\ell+1}}-W_{t_\ell}$. 
The objective  \eqref{eq:param-potential} is then approximated by  
\begin{equation}\label{eq:param-potential-approx}
            \Phi_M(\theta)=\frac{1}{M}\sum_{m=1}^M \bigg[\sum_{\ell=0}^{P-1}F(X_{t_\ell}^{\theta,(m)},\phi^{\theta}(t_\ell,X_{t_\ell}^{\theta,(m)})) \Delta_\ell+G(X_{T}^{\theta,(m)})\bigg],
\end{equation}
where 
$\left(X^{\theta,(m)}\right)_{ m\in[M]}=\left(X_{i }^{\theta_i,(m)} \right)_{ i\in[N] ,m\in[M]}$ are $M$ independent trajectories of   \eqref{eq:param-dynamics-disc}.
Note that when the system is deterministic (i.e., $\sigma_i=0$ for all $i \in [N]$), it suffices  to take $M=1$ in \eqref{eq:param-potential-approx}.
The derivative $\partial_\theta \Phi_M(\theta)$ can be computed via automatic differentiation.


   

\section{Numerical Experiments} 
\label{sec:numeric}

Now we aim to evaluate the performance of our proposed model, including the following research questions: 
\begin{itemize}[leftmargin=*, itemsep=0pt, parsep=0pt]
    \item \textbf{RQ1:}  How would interaction terms affect collision avoidance with other cars? 
    \item \textbf{RQ2:} Could CAVs avoid collisions with obstacles under 2D control?
    \item \textbf{RQ3:} How do different types of vehicles with various interactions affect the equilibrium?
\end{itemize}

\textbf{Experimental setup.}
The experiments use a fleet of CAVs on a long highway (1D) or its 2D extension with a common start position ($x_{0,i}=-1$ in 1D cases and $(-1,-1)$ in 2D cases) and a target position ($z_{i}=1$ in 1D cases and $(1,1)$ in 2D cases). 

\textbf{Procedure.}
We apply the policy-gradient algorithm (Sec.~\ref{sec:algo}) to three scenarios: 
(i) a 1D game comparing weak and strong interactions under both \emph{velocity} vs. \emph{acceleration} control models; 
(ii) a 2D acceleration-control game with obstacles of different sizes to assess how CAVs avoid collisions to obstacles; 
and (iii) a 1D heterogeneous game with \emph{large/medium/small} vehicle types, implemented via size-dependent cross-group interaction strengths.




\subsection{Interaction regimes in cost function}
\label{sec:numerical_interaction}
This section examines the impact of $\beta$ in \eqref{eq:interaction_kernel} for both velocity-control and acceleration-control models. Specifically, for each model, we test   two   choices  $\beta=0$ and $\beta=1$, which characterize the weak-interaction regime and strong-interaction regime, respectively. We consider a 10-player game in CAV control and assume that each player has a one-dimensional state, which follows from \eqref{eq:CAV_1} under velocity control and 
\eqref{eq:CAV_2} under acceleration control with a uniform noise level $\sigma_i=0.1$. Each vehicle $i$  
minimizes its objective function defined in \eqref{eq:cost_i} with $f_i(x,a_i)=0.1a_i^2$ and  $g_{i}(x)=10|x-z_{i}|^2$, where $z_{i}=1$ is the target terminal position for vehicle $i$.  
Interaction kernel $\lambda_{ij}$ and $K$ 
are defined in \eqref{eq:interaction_kernel} with
$\rho:\sR^d\to [0,\infty)$  given by
$
\rho (z)={(|z|^2+1)^{-1} }.
$ 

\begin{figure}[htbp]
    \centering
    \includegraphics[width=0.41\textwidth]{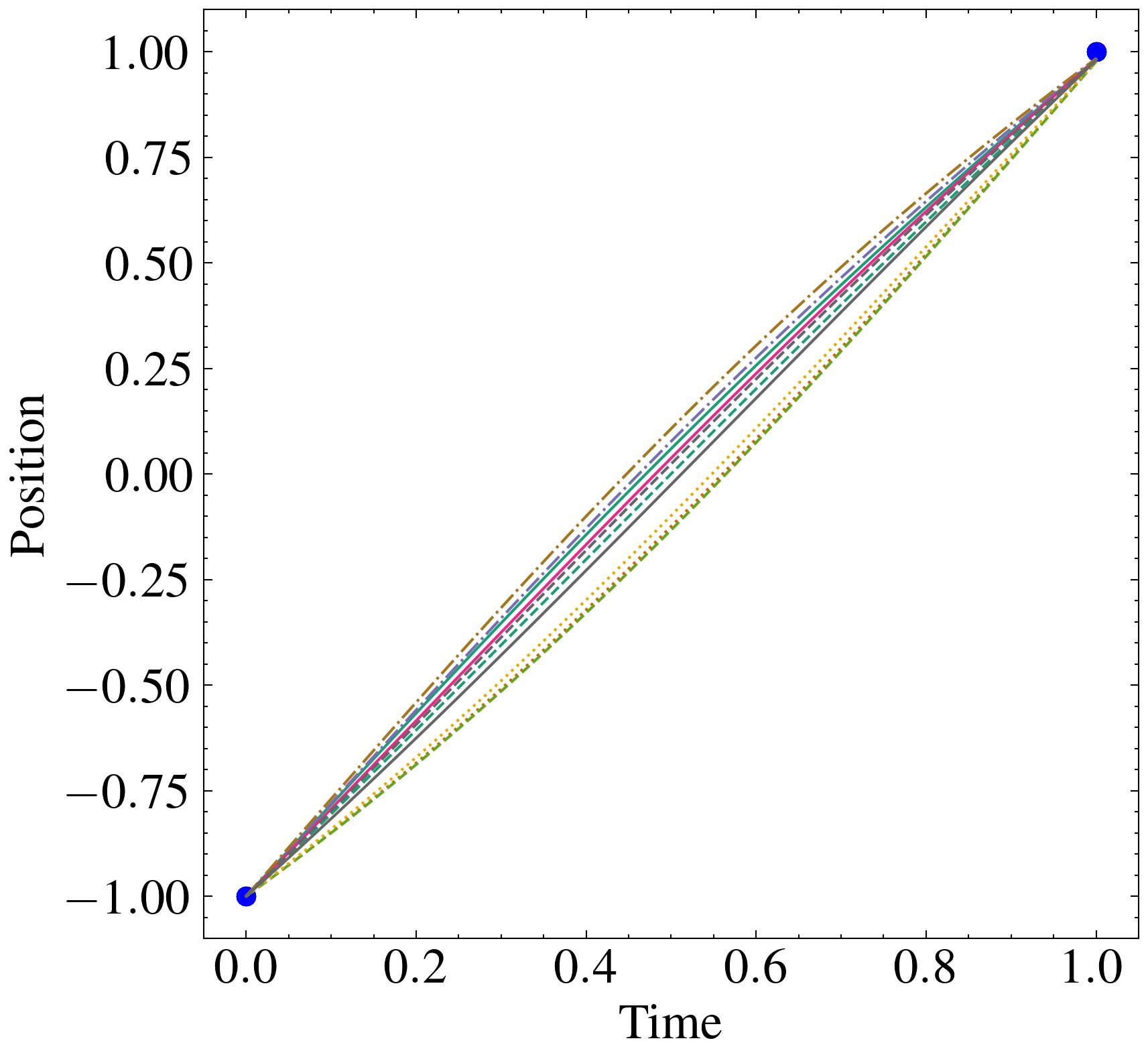}
    \hspace{0.05\textwidth}
    \includegraphics[width=0.4\textwidth]{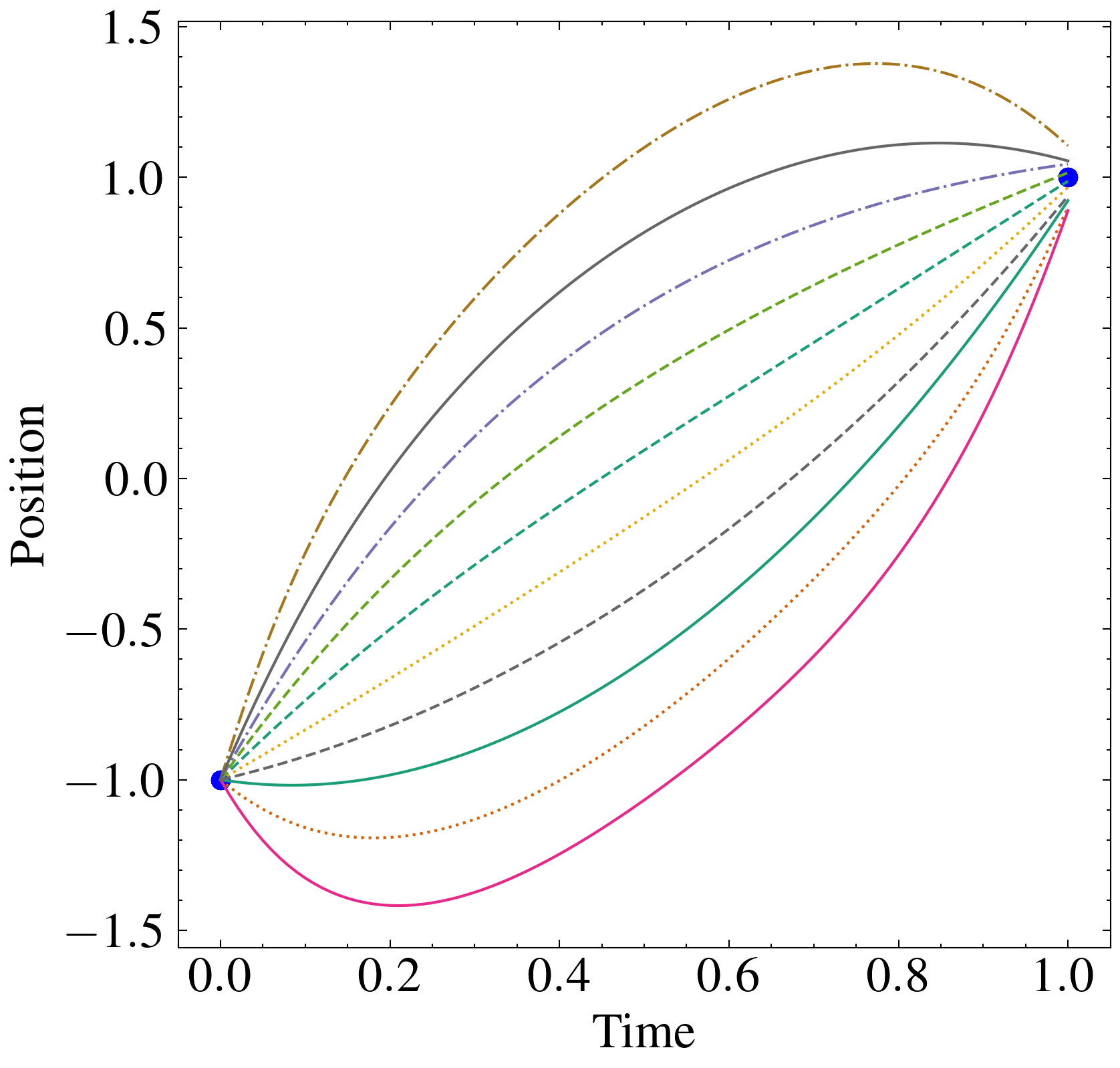}
    \caption{Vehicle trajectories of the control-velocity model for $\beta=0$ (left) and $\beta=1$ (right)}
\end{figure}


\begin{figure}[htbp]
    \centering
    \includegraphics[width=0.41\textwidth]{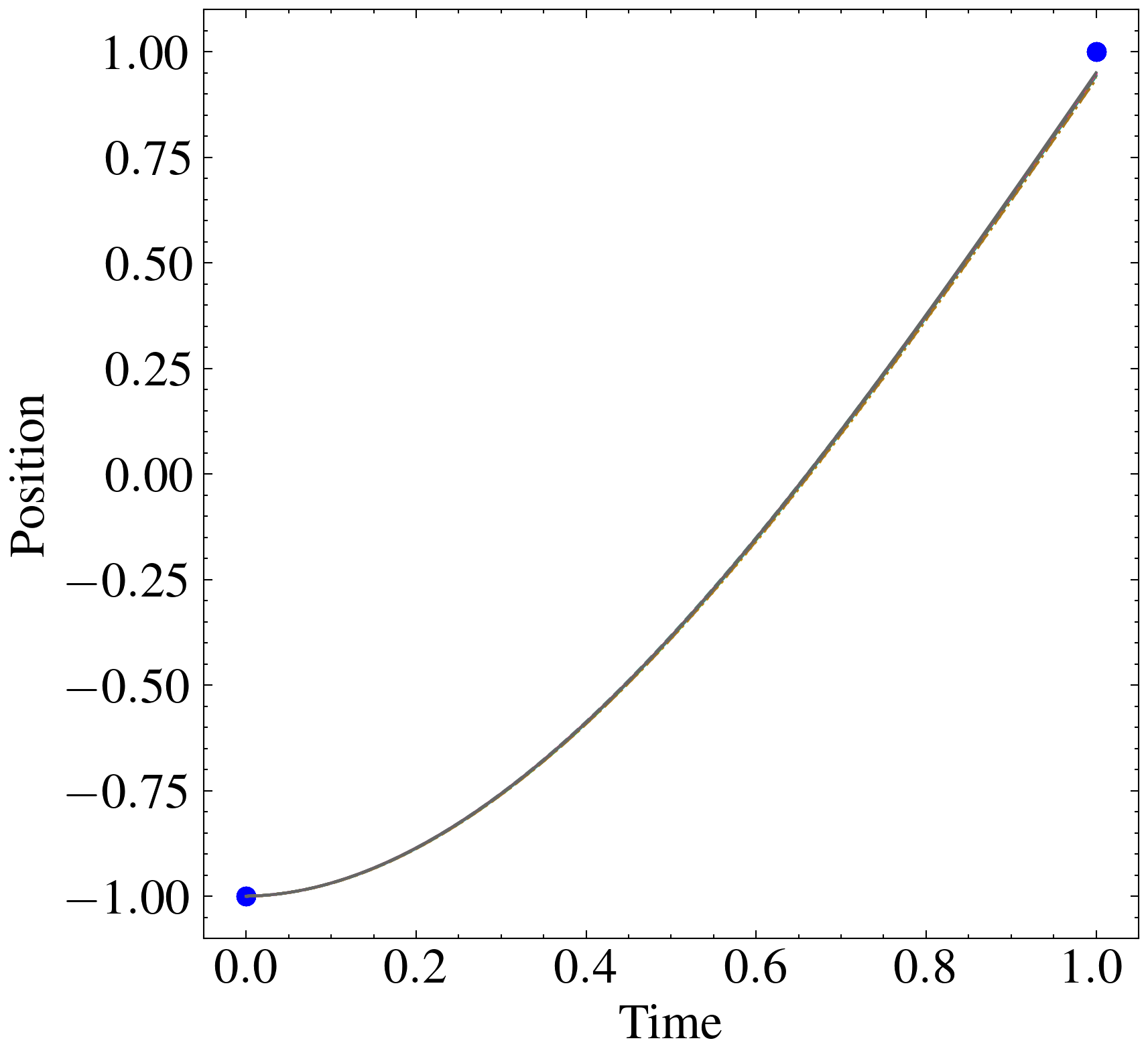}
    \hspace{0.05\textwidth}
    \includegraphics[width=0.4\textwidth]{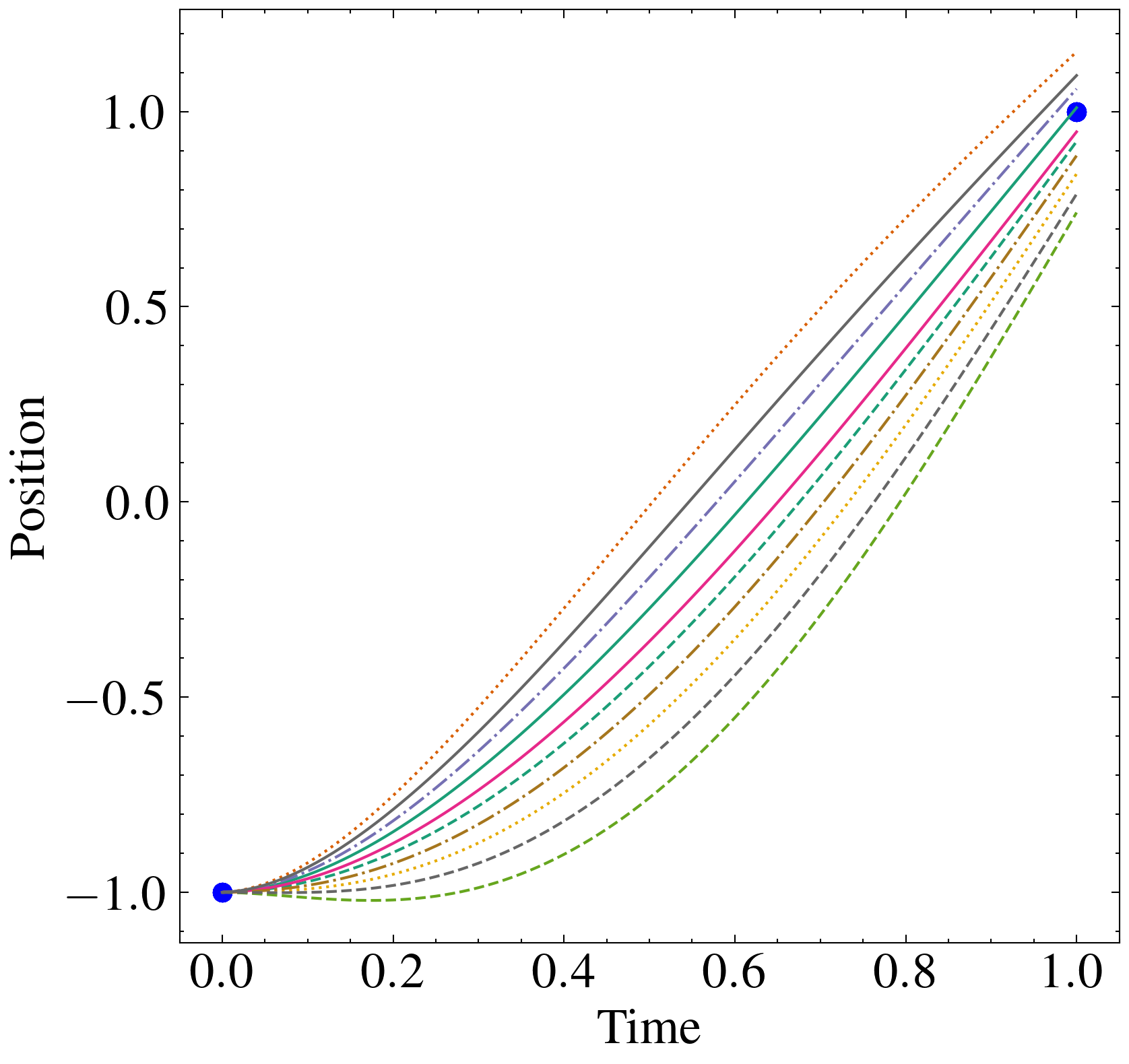}
    \caption{Vehicle trajectories of the control-acceleration model   for $\beta=0$ (left) and $\beta=1$ (right)}
    \label{fig:second_order_1D}
\end{figure}

We observe that, in the weak-interaction regime ($\beta=0$), both the velocity- and acceleration-control models yield near-identical controls and hence similar trajectories, reflecting the symmetric MFG structure in which agents adopt the same feedback policy. However, in the strong-interaction regime ($\beta=1$), vehicles deliberately choose different controls to give each other space, and their paths diverge. This shows that the strong-interaction regime ensures collision avoidance more effectively. 

In addition, it can be seen that under velocity-control, vehicles frequently detour (even briefly backtrack) to avoid others yet still reach the target accurately, whereas under acceleration-control, trajectories are smoother but exhibit dispersed terminal positions, indicating greater difficulty in hitting the target precisely. This happens because velocity control gives agents direct authority over position with cheap late corrections, so they can prioritize separation and fix position near the end; acceleration control introduces inertia, making sharp turns and last-minute corrections costly, so the optimizer smooths motion and tolerates small terminal errors.



\subsection{Two-dimensional acceleration control with obstacles}
This section  examines vehicle behavior in the presence of circular obstacles on the road.
We consider a 10-player game in two-dimensional CAV control, i.e. $d=2, N=10$. The state dynamics follows \eqref{eq:CAV_2} under acceleration control without noise ($\sigma_i \equiv 0$). Each vehicle $i$  minimizes its objective function   \eqref{eq:cost_i},
where
the terminal cost is 
$g_{i}(x)=2|x-z_{i}|^2$ for  $z_{i}=(1,1)$, 
and the interaction kernel $\lambda_{ij}$ and $ K$ are defined the same as in Sec.~\ref{sec:numerical_interaction} with $\beta=1$. 
The running cost $f_i$ depends on the position and the size of the obstacle. 
Specifically, we consider two circular obstacles centered at  $(0,0)$ with radii $0.1$ and $0.5$, respectively. To capture an obstacle avoidance behavior, we consider the running cost
$f_i(x,a_i)=0.02a_i^2+ h(x)$, with 
\begin{equation*}
    h(x)\coloneqq 1000\left(1-\frac{1}{1+\exp[10(1-M|x|^2)]}\right),
\end{equation*}
where $M=100$ for the small obstacle of radius $0.1$
and $M=4$ for   the large obstacle of radius $0.5$.
The cost $h$ is a   differentiable  version of
forbidding passage through the   circular obstacles
\cite{baros2025mean}.
For comparison,  
we also consider the case without obstacles by setting
  $h\equiv 0$.


\begin{figure}[htbp]
    \centering
    \includegraphics[width=0.3\textwidth]{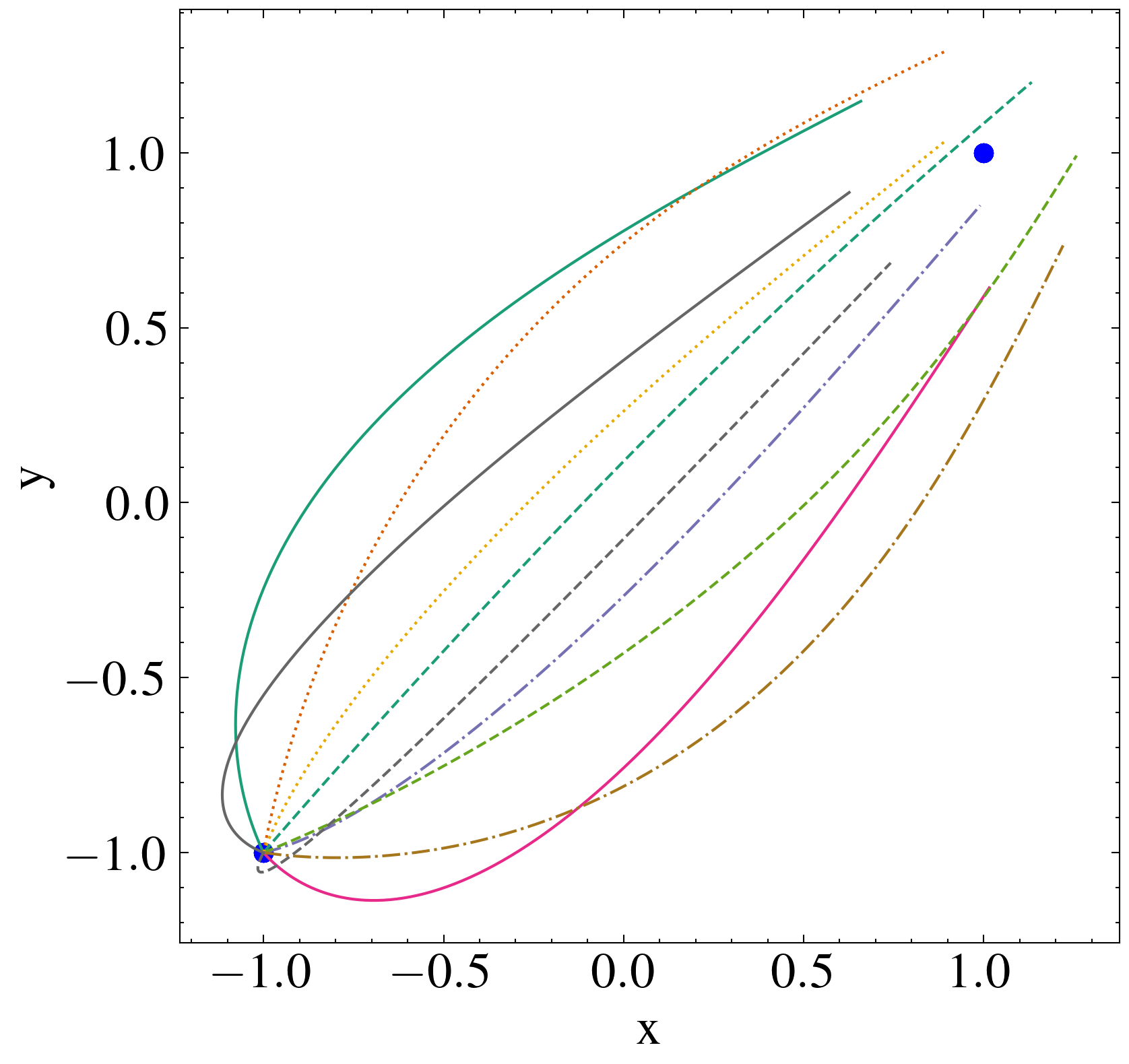}
    \hspace{0.03\textwidth}
    \includegraphics[width=0.3\textwidth]{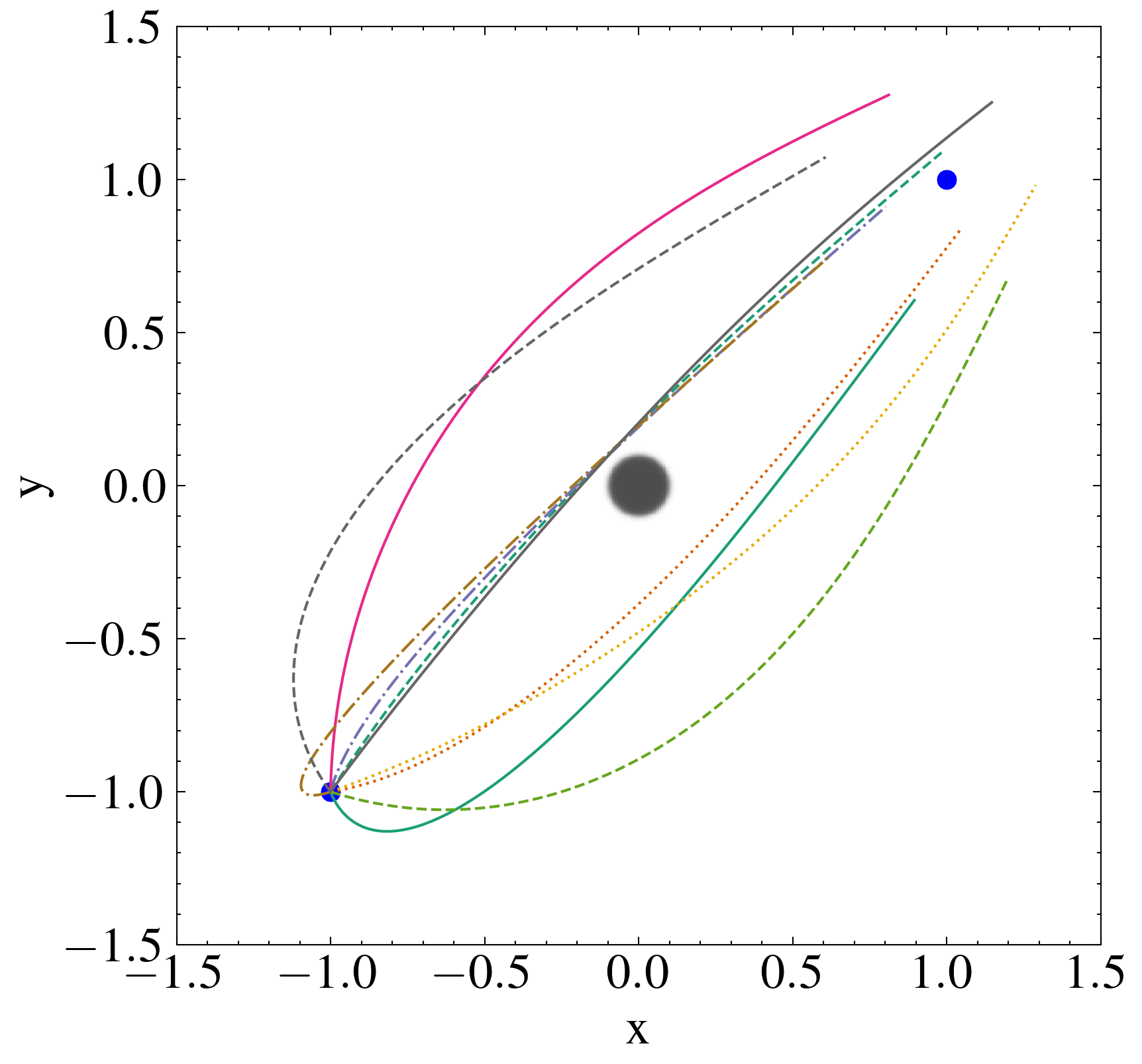}
    \hspace{0.03\textwidth}
    \includegraphics[width=0.3\textwidth]{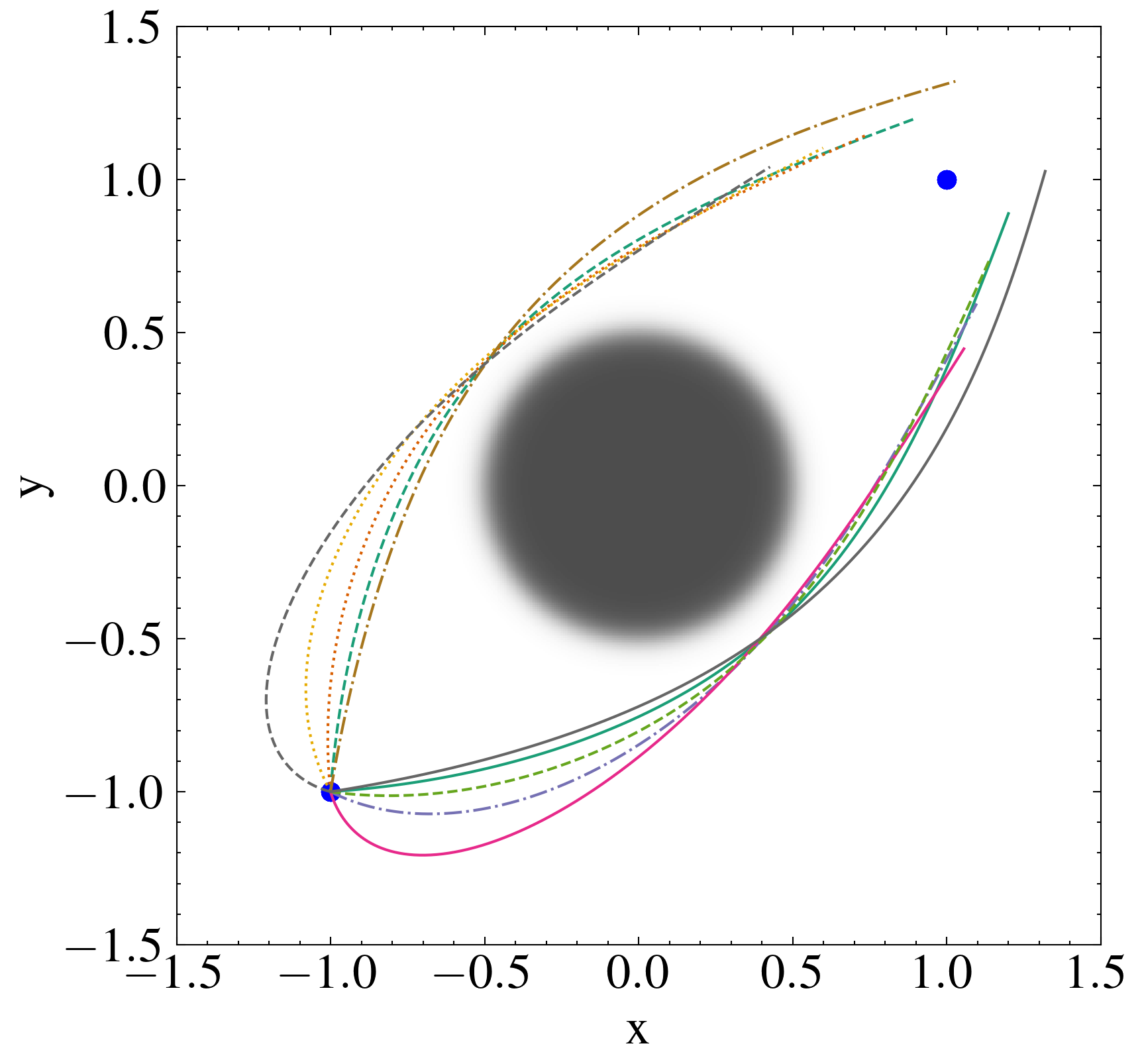}
    \caption{Vehicle trajectories under acceleration-control 
    without obstacle (left), with a small obstacle (middle), and with a large obstacle (right)}
    \label{fig:small_obstacle}
\end{figure}


Fig.~\ref{fig:small_obstacle} shows that    introducing  an obstacle causes the vehicle trajectories to adjust accordingly  to satisfy obstacle-avoidance constraints. When a large obstacle is present, its wider exclusion region forces the trajectories to deviate earlier and more noticeably from their nominal paths. This results in a global reshaping of the motion patterns. In contrast, a smaller obstacle influences the trajectories only locally: only vehicles passing near its vicinity exhibit visible adjustments, while others remain close to the baseline case. These observations highlight that the trajectories respond adaptively to the obstacle, and the extent of deviation is determined by both the size and placement of the obstacle within the environment.

\subsection{Heterogeneous vehicle types}
\label{sec:different_vehicle}

In this section, we consider the games with heterogeneous vehicles and investigate how different vehicles behave in the game. We take a 9-player game in a 1D CAV game without noise, where vehicles are divided into three groups: large, medium, and small. Each group contains three vehicles, with size parameters $\tau=10$ for large vehicles, $\tau=3$ for medium vehicles, and $\tau=0.5$ for small vehicles. For each vehicle, we sample $\gamma_i\tau_i$ from $\text{Unif}(0.8,1.2)$ which introduces randomness while ensuring that larger vehicles exert greater influence on others and are, in turn, less affected by them. The parameters of $\gamma_i$ and $\tau_i$ are given in   Table \ref{table: vehicle types parameters}.
\begin{table}[h]
\centering
\caption{Values of $\{\gamma_i\}_{i\in[9]}$ and $\{\tau_i\}_{i\in[9]}$.}\label{table: vehicle types parameters}
\begin{tabular}{l c c c c c c c c c}
\toprule
Vehicle & 1 & 2 & 3 & 4 & 5 & 6 & 7 & 8 & 9 \\
\midrule 
$\gamma_i$ & 0.115 & 0.117 & 0.095 & 0.395 & 0.319 & 0.347 & 1.805 & 2.235 & 2.353 \\
$\tau_i$ & 10 & 10 & 10 & 3 & 3 & 3 & 0.5 & 0.5 & 0.5 \\
\bottomrule
\end{tabular}
\end{table}
 Each vehicle minimizes the objective function  \eqref{eq:cost_i} with  $f_i(x,a_i)=0.02a_i^2$,  $g_{i}(x)=2|x-1|^2$,
$\lambda_{ij}=\gamma_i\tau_j$ for all $i,j\in[9]$ and  $K(z)=(N^2|z|^2+1)^{-1}$, which extends \eqref{eq:interaction_kernel} by incorporating    asymmetric interactions.
%

\begin{figure}[htbp]
    \centering
    \includegraphics[width=0.4\textwidth]{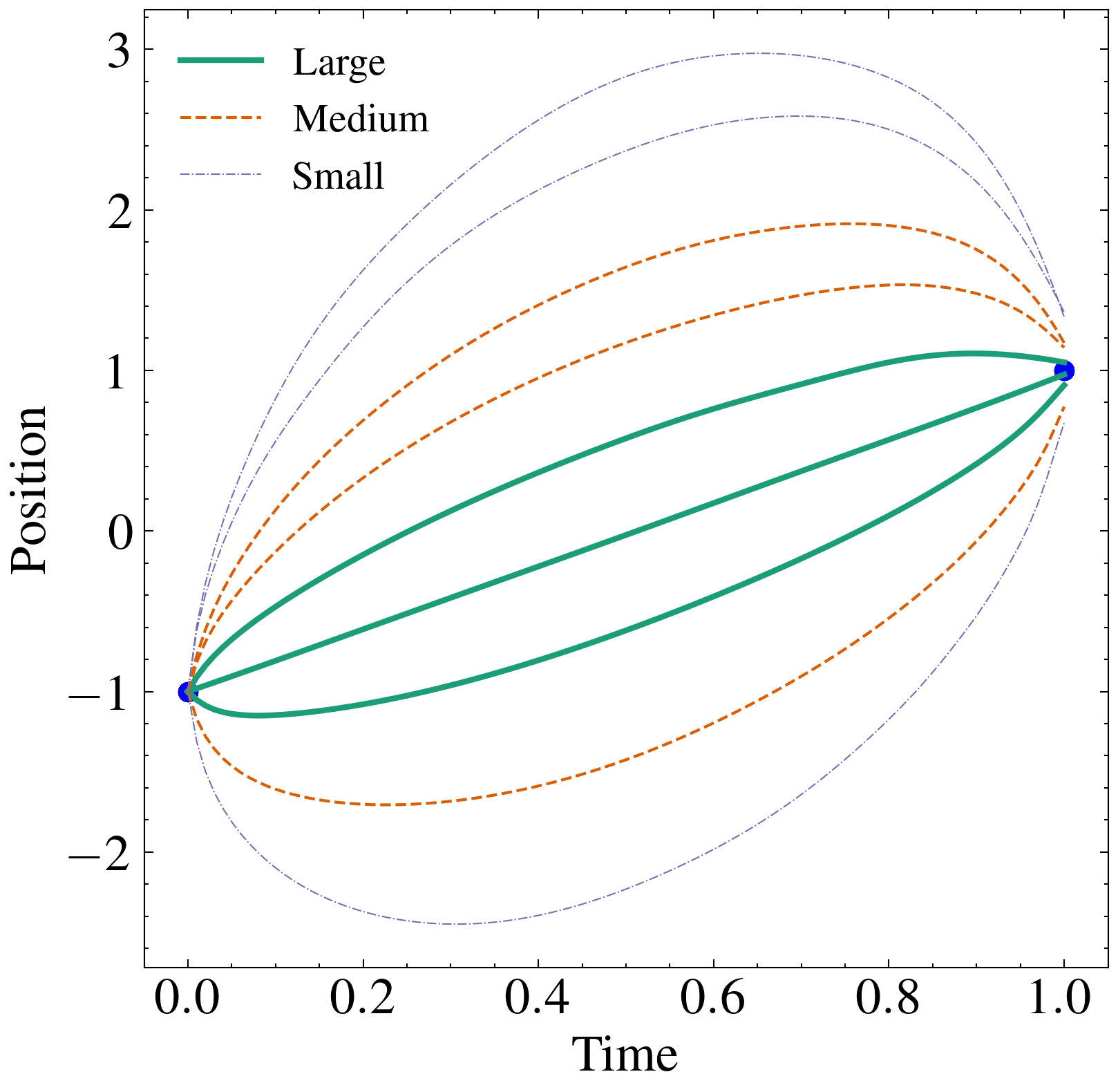}
    \hspace{0.05\textwidth}
    \includegraphics[width=0.418\textwidth]{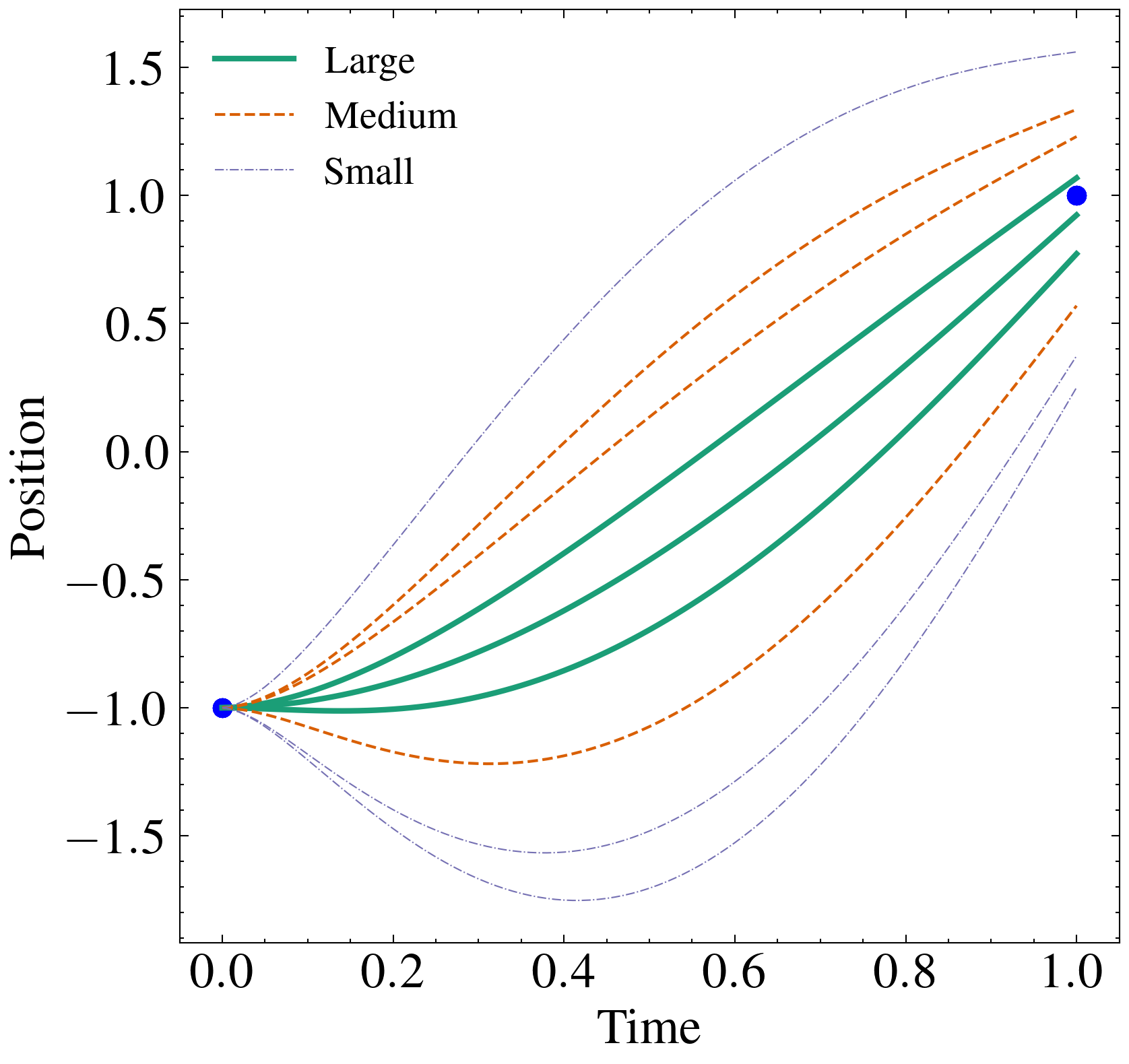}
    \caption{Vehicle trajectories of different types under control-velocity (left) and control-acceleration (right)}
    \label{fig:different_vehicle}
\end{figure}
Fig.~\ref{fig:different_vehicle} 
shows that both   models exhibit a consistent hierarchical structure among the large, medium, and small vehicles. The trajectories follow a clear ordering, with the large vehicles occupying the central path, the medium vehicles positioned outward from it, and the small vehicles pushed even further toward the periphery. This ordering reflects the asymmetric interaction strength among different vehicle sizes: larger vehicles impose a stronger impact on the motion of nearby vehicles and exhibit a lower sensitivity to the actions of others. In contrast, smaller vehicles are more strongly affected by the presence of larger ones. This size-dependent interaction pattern appears consistently in both the velocity- and acceleration-control models.

\section{Conclusions} 
\label{sec:conclude}

We proposed an $\alpha$-potential game
for decentralized control of heterogeneous CAVs. 
Scalable policy gradient algorithm is developed for computing $\alpha$-NEs with decentralized neural-network policies.
A series of numerical experiments are performed, demonstrating that our model  effectively captures collision avoidance and agent heterogeneity.



\section{Acknowledgement} 
\label{sec:acknowledgement}
Di acknowledges support from NSF CMMI-1943998.
Zhang acknowledges support from the Imperial Global Connect Fund.
 This work was partially supported by a grant from the Simons Foundation.
 Zhang  thanks the Isaac Newton Institute for Mathematical Sciences, Cambridge, for support and hospitality during the programme Bridging Stochastic Control and Reinforcement Learning, where work on this paper was undertaken. This work was supported by EPSRC grant no EP/R014604/1.
 This research was funded in part by JPMorgan Chase \& Co. Any views or opinions expressed herein are solely those of the authors listed, and may differ from the views and opinions expressed by JPMorgan Chase \& Co. or its affiliates. This material is not a product of the Research Department of J.P. Morgan Securities LLC. This material does not constitute a solicitation or offer in any jurisdiction.


\bibliographystyle{siam}
\bibliography{pg1.bib}

 \end{document}